\documentclass[a4paper]{amsart}
\usepackage{amssymb}
\usepackage[colorlinks = true, citecolor = blue]{hyperref} 
\newcommand{\C}{\mathbb{C} }
\newcommand{\R}{\mathbb{R} }
\newcommand{\Z}{\mathbb{Z} }
\newcommand{\T}{\mathbb{T} }

\newcommand{\N}{\mathbb{N} }

\newcommand{\Zmod}[1]{\Z/#1\Z}
\newcommand{\unitsmod}[1]{(\Z/#1\Z)^\times}
\newcommand{\latticepoint}{m}
\newcommand{\toruspoint}{\xi}
\newcommand{\latticesubset}{F}

\newcommand{\fxn}{f}

\newcommand{\lpnorm}[2]{\left\| #2 \right\|_{\ell^{#1}(\Z^\dimension)}}
\newcommand{\disup}[3]{\sup_{#1 \leq #2 < 2#1} #3}

\newcommand{\inparentheses}[1]{\left( #1 \right)}
\newcommand{\inbrackets}[1]{\left[ #1 \right]}
\newcommand{\inbraces}[1]{\left\{ #1 \right\}}
\newcommand{\sizeof}{\#}

\newcommand{\absolutevalueof}[1]{\left\lvert #1 \right\rvert}

\newcommand{\eof}[1]{e \left({#1} \right)} 

\newcommand{\arithmeticFT}[1]{\widehat{#1}}
\newcommand{\latticeFT}[1]{\widehat{#1}}
\newcommand{\torusFT}[1]{\widehat{#1}}
\newcommand{\contFT}[1]{\widetilde{#1}}
\newcommand{\indicator}[1]{\mathbf{1}_{#1}}

\newcommand{\convolvedwith}{*}

\newcommand{\dimension}{d}
\newcommand{\degree}{k}

\newcommand{\unit}{a}
\newcommand{\modulus}{q}
\newcommand{\bigmodulus}{Q}
\newcommand{\Fareyfraction}{\unit/\modulus}

\newcommand{\majorarcs}{\mathfrak{M}}
\newcommand{\minorarcs}{\mathfrak{m}}
\newcommand{\twistedGausssum}[1]{G(\unit,\modulus,#1)}
\newcommand{\acceptableradii}{\mathcal{R}_{\degree,\dimension}} 
\newcommand{\numberoflatticepoints}{N_{\degree,\dimension}}

\newcommand{\hypothesis}[1]{H_{\degree}\left(#1 \right)}
\newcommand{\meanvaluehypothesis}[1]{MVH_{\degree}\left(#1 \right)}
\newcommand{\powerloss}{\theta} 
\newcommand{\singlesavings}{\omega} 
\newcommand{\dyadicsavings}{\nu} 

\newcommand{\error}{{E}_{\radius}}

\newcommand{\HLmultiplier}{\arithmeticFT{C_\radius}}
\newcommand{\HLop}{C_{\radius}}

\newcommand{\fracHLmultiplier}{\arithmeticFT{C_{\radius}^{\Fareyfraction}}}
\newcommand{\lowfracHLmultiplier}{\arithmeticFT{C_{\radius}^{\Fareyfraction, low}}}
\newcommand{\highfracHLmultiplier}{\arithmeticFT{C_{\radius}^{\Fareyfraction, high}}}
\newcommand{\fracHLop}{C_{\radius}^{\Fareyfraction}}

\newcommand{\lowfracHLop}{C_{\radius}^{\Fareyfraction, low}}
\newcommand{\highfracHLop}{C_{\radius}^{\Fareyfraction, high}}
\newcommand{\lowfracHLkernel}{K_{\radius}^{\Fareyfraction, low}}
\newcommand{\lowHLop}{C_{\radius}^{low}}
\newcommand{\highHLop}{C_{\radius}^{high}}

\newcommand{\dyadicradius}{R}

\newcommand{\radius}{r}
\newcommand{\bigradius}{R}
\newcommand{\radii}{\mathcal{R}}
\newcommand{\altitude}{\alpha}

\newcommand{\arithmetichigherorderspheremeasure}{\sigma_\radius}
\newcommand{\conthigherordersphere}{\mathcal{S}_{\radius}^{\degree,\dimension}} 
\newcommand{\conthigherorderspheremeasure}{d\sigma_\radius}
\newcommand{\surfacemeasure}{\sigma}

\newcommand{\operatornorm}[2]{\left\Arrowvert #2 \right\Arrowvert_{#1}}
\newcommand{\ellpoperatornorm}[1]{\operatornorm{\ell^p(\Z^\dimension) \to \ell^p(\Z^\dimension)}{#1}}
\newcommand{\elltwooperatornorm}[1]{\operatornorm{\ell^2(\Z^\dimension) \to \ell^2(\Z^\dimension)}{#1}}

\newcommand*{\discreteHLmaxop}{M_*}
\newcommand{\averagefxn}{A_{\radius}\fxn} 
\newcommand{\avgop}{A_{\radius}} 
\newcommand{\maxfxn}{M_\radii \fxn}
\newcommand{\maxop}{M_\radii}
\newcommand{\fullmaxop}{M_{\acceptableradii}}

\newcommand*{\setof}[1]{\left\{#1\right\}}

\newcommand*{\density}{\delta}
\DeclareMathOperator{\Dil}{Dil}
\newcommand*{\dilateby}[1]{\Dil_{#1}}
\newcommand*{\bigfrequencydilation}{\Delta}

\newcommand{\arithmeticsphere}[1]{S^{\degree,\dimension}_{#1}} 
\newcommand{\arithmetichigherordersphere}{S^{\degree,\dimension}_{\radius}} 
\newcommand{\moment}{s}

\newtheorem{proposition}{Proposition}[section]
\newtheorem{lemma}{Lemma}[section]
\newtheorem{theorem}{Theorem}
\newtheorem{corollary}{Corollary}

\newtheorem*{MSW}{Magyar--Stein--Wainger}{\bf}{\it}
{\bf}{\it}
\newtheorem*{Ionescu}{Ionescu}{\bf}{\it}
\newtheorem*{Wooley}{Wooley}{\bf}{\it}
\newtheorem*{Steckin}{Ste{\v{c}}kin's estimate}{\bf}{\it}
\newtheorem*{approximationlemma}{The Approximation Formula}{\bf}{\it}
\newtheorem*{SupHypothesis}{Sup Hypothesis $\hypothesis{\powerloss}$}{\bf}{\it}
\newtheorem*{MeanValueHypothesis}{Mean Value Hypothesis $\meanvaluehypothesis{\dimension}$}{\bf}{\it} 
\theoremstyle{definition}
\newtheorem{definition}{Definition}[section]

\theoremstyle{remark}
\newtheorem{remark}{Remark}[section]
\title[Estimates for discrete k-spherical maximal functions]
{ 
Restricted weak-type endpoint estimates for discrete k-spherical maximal functions}
\author{
Kevin Hughes 
}
\address{
Kevin Hughes\\
School of Mathematics\\
The University of Edinburgh\\
EDINBURGH\\
Scotland, UK\\
}
\begin{document}

\begin{abstract}
In this paper, we use the Approximation Formula for the Fourier transform of the solution set of lattice points on k-spheres and methods of Bourgain and Ionescu to refine the \(\ell^p(\mathbb{Z}^d)\)-boundedness results for discrete k-spherical maximal functions to a restricted weak-type result at the endpoint. 
Moreover we introduce a novel Approximation Formula for a single average; this allows us to improve our bounds for discrete k-spherical maximal functions along sparse subsequences of radii by exploiting recent progress of Wooley on the Vinogradov mean value conjectures. 
In particular we have improved bounds for lacunary discrete k-spherical maximal functions when \(k \geq 3\). 
We introduce a density-parameter, which may be viewed as a discrete version of Minkowski dimension used in related works on the conitnuous phenomena, to prove our results for sparse averages. 
%
\end{abstract}

\maketitle 
\tableofcontents 

\newpage
\section{Introduction} 

Fix the degree \(\degree \geq 2\) and dimension \(\dimension\), positive integers such that $\dimension > \degree$. As in \cite{Hughes_Vinogradov}, we define the \emph{arithmetic $\degree$-sphere of radius $\radius$ in $\dimension$ dimensions} as 
\begin{equation}
\arithmetichigherordersphere := \{\latticepoint \in \Z^\dimension : \sum_{i=1}^{\dimension} |\latticepoint_i|^\degree=r^\degree\} 
. 
\end{equation}
The $\degree$-sphere of radius $\radius$ contains $\numberoflatticepoints(\radius) = \# \arithmetichigherordersphere$ lattice points, $\arithmetichigherordersphere$ is possibly non-empty only when $\radius^\degree \in \N$. We denote the set of positive radii $\radius$ such that $\arithmetichigherordersphere$ is not empty by $\acceptableradii$. 
For a function $\fxn : \Z^\dimension \to \C$ and $\radius \in \acceptableradii$, we introduce the \emph{$\degree$-spherical averages}, 
\begin{equation}
\averagefxn(x) 
= \frac{1}{\numberoflatticepoints(\radius)} \sum_{y \in \arithmetichigherordersphere} \fxn(x-y) 
\end{equation} 
for each \(x \in \Z^\dimension\). 
And for a subsequence $\radii \subset \acceptableradii$, we introduce the maximal function over $\radii$ as 
\begin{equation}\label{eq:maxop_definition}
\maxfxn 
= \sup_{\radius \in \radii} \absolutevalueof{\averagefxn} 
. 
\end{equation}

These maximal functions are the arithmetic analogues of continuous maximal functions over $\degree$-spheres in Euclidean space. In the continuous setting, maximal functions associated to compact convex hypersurfaces are bounded on a range of $L^p(\R^\dimension)$ spaces depending on the geometry of the hypersurface and dimensional properties of the collection of radii -- see \cite{Stein_spherical}, \cite{Bourgain_circular}, \cite{BNW}, \cite{DRdF_lacunary}, \cite{STW_endpoint_spherical}, \cite{STW_pointwise_lacunary} and \cite{SteinHA}. 
%
%
In particular we were motivated by the results of \cite{Calderon_lacunary_spherical}, \cite{Coifman_Weiss_lacunary_spherical}, \cite{Seeger_Wainger_Wright} and \cite{Duoandikoetxea_Vargas}. 

\subsection{Previous results}
Motivated by this Euclidean phenomenon, in \cite{Magyar_dyadic}, Magyar initiated the study of arithmetic $\degree$-spherical maximal functions and proved that the dyadic versions of the maximal operator in \eqref{eq:maxop_definition} are uniformly bounded for a range of $\ell^p(\Z^\dimension)$ spaces depending on the degree and dimension. 
%
%
For continuous maximal functions, there is an intrinsic dilation invariance that allows one to deduce that the boundedness of the dyadic maximal operator is equivalent to the full maximal function by the use of Littlewood--Paley theory. However, this argument fails in the arithmetic setting as $\Z^\dimension$ lacks dilation invariance, and a new idea was needed to understand the full maximal function. Building on Magyar's work, Magyar--Stein--Wainger \cite{MSW_spherical} studied the full maximal function for degree $\degree = 2$ proving the arithmetic analogue of Stein's spherical maximal theorem of  \cite{Stein_spherical}. 
\begin{MSW}
Let $\degree=2$ and $\dimension \geq 5$, then $\fullmaxop$ is bounded on $\ell^p(\Z^\dimension)$ for $p>\frac{\dimension}{\dimension-2}$.
\end{MSW}
\noindent
Analogous to the Bourgain's restricted weak-type result for the continuous spherical maximal function in 3 or more dimensions, see \cite{Bourgain}, Ionescu improved the Magyar--Stein--Wainger result by proving the restricted weak-type result at the endpoint $p=\frac{\dimension}{\dimension-2}$ in \cite{Ionescu_spherical}. 
\begin{Ionescu}
If $\degree=2$ and $\dimension \geq 5$, then $\fullmaxop$ is bounded from $\ell^{\frac{\dimension}{\dimension-2},1}_{rest}(\Z^\dimension)$ to $\ell^{\frac{\dimension}{\dimension-2},\infty}(\Z^\dimension)$.
\end{Ionescu}
%

\noindent 
These results are sharp, except possibly for removing the restricted assumption in Ionescu's result. 
%

In another direction, \cite{Magyar_ergodic} extended the results of Magyar--Stein--Wainger to positive definite, nondegenerate, homogeneous integral forms to prove boundedness of the corresponding maximal operator on $\ell^2(\Z^\dimension)$ and pointwise convergence of their ergodic averages when $\dimension>(\degree-1)2^\degree$; this includes the family of $\degree$-spheres considered here. 
%
%
The motivating problem for this paper is to sharpen Magyar's results by obtaining the full range of \(\ell^p(\Z^\dimension)\)-spaces on which \(\fullmaxop\) is bounded. 
Using refined estimates for exponential sums coming from recent progress in Waring's problem, see \cite{Wooley_efficient_congruencing1} or \cite{Wooley_best_in_June2014}, and oscillatory integrals, see \cite{BNW}, the author improved Magyar's results in \cite{Hughes_Vinogradov} for the family of $\degree$-spheres to 
\setcounter{theorem}{-1}
\begin{theorem}\label{theorem:Hughes_thesis_bound}
If $\degree \geq 3$ and $\dimension>2\degree^2(\degree-1)$, then $A_{\radii_{\degree,\dimension}}$ is bounded on $\ell^p(\Z^\dimension)$ for $p>\frac{\dimension}{\dimension - \degree^2(\degree-1)}$.
\end{theorem}
%
In this paper, we refine Theorem~\ref{theorem:Hughes_thesis_bound} to a restricted weak-type endpoint result and deduce further bounds for thin sequences.

\subsection{Summary of results}

For each degree \(\degree \geq 3\), there is much room to improve Theorem~\ref{theorem:Hughes_thesis_bound} in the range of \(p\) and the range of dimension \(\dimension\). 
The expected range of dimensions is difficult to describe; we refer the reader to Conjectures~1 and 2 of \cite{Hughes_Vinogradov}. 
The range of \(p\) is easier to predict in sufficiently high dimensions. 
Assume that the dimension \(\dimension\) is sufficiently large with respect to the degree \(\degree\) so that \(\numberoflatticepoints(\radius) \eqsim \radius^{\dimension-\degree}\) for all sufficiently large \(\radius \in \acceptableradii\). 
Testing the maximal operator on the Dirac delta function, we see that the maximal operator over \(\acceptableradii\), \(\fullmaxop\), fails to be bounded on \(\ell^p(\Z^\dimension)\) for \(p \leq \frac{\dimension}{\dimension-\degree}\), but we expect that \(\fullmaxop\) is bounded on $\ell^p(\Z^\dimension)$ for $p>\frac{\dimension}{\dimension-\degree}$. 

In this paper we improve the range of \(\ell^p(\Z^\dimension)\)-boundedness for \(\maxop\), the \(\degree\)-spherical maximal function over a subset \(\radii \subset \acceptableradii\) when the `size of \(\radii\)' is small. 
Our results are phrased in terms of two hypotheses: \(\hypothesis{\powerloss}\) and \(\meanvaluehypothesis{\dimension}\) which we describe in this section.
Our first hypothesis is a bound on the supremum of exponential sums near rational points. 
\begin{SupHypothesis}
For all \(N \in \N\), suppose that there exists integers $ 1 \leq \unit < \modulus \leq N$ relatively prime such that $\absolutevalueof{t-\unit/\modulus}~\leq~\modulus^{-2}$. Then
\[ 
\sum_{n=1}^N e(t n^\degree + \xi n) 
\lesssim N (\modulus^{-1} + N^{-1} + \modulus N^{-\degree})^\powerloss 
\]
with implicit constants independent of $\xi$, $\modulus$ and $N$.
\end{SupHypothesis}
\noindent 
This hypothesis differs slightly from the hypothesis in \cite{Hughes_Vinogradov}. There, the hypothesis includes a logarithmic-loss; essentially replacing our $\hypothesis{\powerloss}$ below by $\hypothesis{\powerloss-\epsilon}$ for all $\epsilon>0$. This allows us to strengthen our results abstractly. In practice, the exponential sum bounds that we can plug into our hypothesis come with a log-loss so that we can only recover the same bounds for the full maximal function as in \cite{Hughes_Vinogradov}. 
Note that our Sup Hypothesis is better the larger we may take \(\powerloss \in (0,1) \). 

In contrast our method allows for us to obtain new results for lacunary maximal functions for $\degree \geq 3$, and more generally for maximal functions over thin subsequences of radii. 
%
%
\begin{theorem}\label{theorem:main_theorem}
Assume that the degree $\degree \geq 3$ and $\hypothesis{\powerloss}$ is true for some $0 < \powerloss < 1/2$. 
If $\radii \subset \acceptableradii$ is a subsequence with density-parameter at most $\density$, then the associated maximal function, \(\maxop : \ell^{p,1}(\Z^\dimension) \to \ell^{p,\infty}(\Z^\dimension) \) for 
\(
p = \max 
\setof{
\frac{\dimension}{\dimension-\degree}, 
1 + \frac{\density\degree}{2(\dimension-\degree[\degree+2])+\density\degree},  
1+\frac{\density}{2(\dimension \powerloss - \degree)+\density} 
}
\) 
and 
\( \dimension > \max{\{ \degree(\degree+2), \degree/\powerloss\}} \). 
\end{theorem}
\noindent
See Section~\ref{section:main_theorem} for the definition of density-parameter. 
Theorem~\ref{theorem:main_theorem} specializes to two corollaries. The first improves the $\ell^p(\Z^\dimension)$-boundedness for the full $\degree$-spherical maximal function while the second improves the range for subsequences below a critical density-parameter. 
\begin{corollary}\label{corollary:weak_vinogradov} 
Fix the degree $\degree \geq 3$. If $\hypothesis{\powerloss}$ is true for some \( \powerloss \in (0, [\degree(\degree+2)]^{-1})\), then $\fullmaxop$ is of restricted weak-type $\inparentheses{\frac{\dimension}{\dimension -\degree/2\powerloss},\frac{\dimension}{\dimension -\degree/2\powerloss}}$ for 
$\dimension > \max{\{ \degree(\degree+2), \degree/\powerloss\}}$.
\end{corollary}
\begin{corollary}\label{corollary:lacunary_theorem}
If $\radii \subset \acceptableradii$ is a subsequence with density-parameter at most 
\(
\density 
\leq 
\min 
\{ 
\frac{2 (\dimension-\degree[\degree+2])}{\dimension-2\degree}, 
\frac{2 \degree (\dimension \powerloss - \degree)}{\dimension-2\degree} 
\} 
\), then the associated maximal function, $\maxop$ is restricted weak-type $\inparentheses{\frac{\dimension}{\dimension-\degree},\frac{\dimension}{\dimension-\degree}}$ for 
\( 
\dimension > \max{\{ \degree(\degree+2), \degree/\powerloss\}} 
\) . 
In particular, this holds if $\radii$ is a lacunary subsequence of \(\acceptableradii\). 
\end{corollary}

\begin{remark}
For any degree $\degree \geq 2$, it was conjectured that, the maximal function $\maxop$ is bounded on $\ell^p(\Z^\dimension)$ for all $1<p \leq \infty$, and possibly weak-type (1,1), for any lacunary subsequence $\radii \subset \acceptableradii$. 
This was disproven by J. Zienkiewicz (personal communication), who demonstrated that there are arbtitrarily thin infinite sequences \(\radii \subset \acceptableradii\) such that when \(\degree = 2\),  \(\maxop\) fails to be bounded on \(\ell^p(\Z^\dimension)\) for \(p < \frac{\dimension}{\dimension-1}\). 
These examples extend to higher degrees \(\degree \geq 2\). 
On the other hand when \(\degree=2\) the author, in \cite{Hughes_lacunary_2014}, recently showed that for the maximal function \(\maxop\) over the super-lacunary sequence \(\radii := \setof{\radius_j > 0 : \radius_j^2 = 1+\prod_{i=1}^{h(j)} \mathfrak{p}_j}\) where \(h(j) := 2^{j^v}\) for some \(v>1\) and $\mathfrak{p}_j$ is the $j^{th}$ prime, is strong type \(\ell^{\frac{\dimension}{\dimension-2}}\). 
\end{remark}

The proof of Theorem~\ref{theorem:main_theorem} follows Ionescu's argument in \cite{Ionescu_spherical} which itself is based on Bourgain's method for the continuous analogues of our averages in \cite{Bourgain}. Bourgain's strategy is to decompose our operator into two sublinear pieces, one with a good bound on $\ell^2$ and the other with a bad bound on $\ell^1$, and arbitrage these for an improvement in the middle. 
%
%
Ionescu decomposes the spherical operators into 5 pieces and optimizes their bounds altogether. 
Our argument is a variant of Ionescu's. It differs from Ionescu's by partitioning his argument into two steps. The first step, treated in Section~\ref{section:main_term}, uses Bourgain's strategy to prove the weak-type bound for the main term at the endpoint \(p=\frac{\dimension}{\dimension-\degree}\). The second step, treated in Section~\ref{section:main_theorem} combines the result of the first step with Bourgain's strategy applied to the error term. 
Our use of the union bound appears to be a novelty in our variant of Bourgain's method. Despite its simplicity this application of the union bound is efficient and appears to be over-looked in the literature regarding lacunary spherical averages. In particular, one can give simpler proofs of the Calder\'on and Coifman--Weiss theorems proving that the lacunary spherical maximal function maps \(L^p(\R^\dimension)\) to itself for \(1 < p \leq \infty\). 
%

However, our main novelty is in The Approximation Lemma. By exploiting Vinogradov's mean value theorems and conjectures, we obtain a new bound for The Approximation Lemma in \cite{Hughes_Vinogradov} which controls the error term associated to an individual averages as opposed to a dyadic range. Again, using Bourgain's method a la Ionescu and the union bound, our next results improve the range of boundedness for lacunary \(\degree\)-spherical maximal functions in terms of dimension by a factor of the degree \(\degree\). 
Like Theorem~\ref{theorem:main_theorem} we would like our theorems to understand the precise necessary arithmetic input. 
To this effect, our next results are stated in terms of a mean value hypothesis \(\meanvaluehypothesis{\dimension}\) motivated by Waring's problem. 
\begin{MeanValueHypothesis} 
For a fixed degree \(\degree \geq 3\), we will say that \(\meanvaluehypothesis{\dimension}\) is true for some dimension \(\dimension \in \N\) with \(\dimension > \degree\) if 
\begin{equation}
\# \setof{m_1, \dots, m_\dimension ; n_1, \dots, n_\dimension \in [r] : \sum_{i=1}^\dimension |n_i|^\degree = \sum_{i=1}^\dimension |m_i|^\degree} 
\lesssim 
\radius^{2\dimension-\degree} 
\end{equation} 
for all \(\radius \in \acceptableradii\) as \(\radius \to \infty\) where the implicit constants are assumed to be independent of \(\radius \in \acceptableradii\). 
\end{MeanValueHypothesis}

\begin{theorem}\label{theorem:low_density_improvement_by_mean_values}
Assume that \(\meanvaluehypothesis{\moment}\) is true for some \(\moment > \degree\) and that \(\hypothesis{\powerloss}\) is true for some \(\powerloss \in (0,1)\). 
If \(\radii\) is a subsequence of \(\acceptableradii\) with density-parameter at most \(\density \leq \powerloss[\dimension-2\moment]\), then \(\maxop\) is bounded from \(\ell^{p,1}(\Z^\dimension)\) to \(\ell^{p,\infty}(\Z^\dimension)\) for 
\(
p = 
\max 
\{
\frac{\dimension}{\dimension-\degree}, 
1 + \frac{\density\degree}{2(\dimension-\degree[\degree+2])+\density\degree},  
1+\frac{\density}{2\powerloss[\dimension-2\moment] - \density} 
\}
\) 
and \(\dimension > \max \{2s, \degree(\degree+2)\}\). 
\end{theorem}

In Section~\ref{section:connection_to_Vinogradov} we connect our hypothesis to Waring's problem and the Vinogradov mean value conjectures and give explicit estimates in terms of the currently best known bounds. 


\subsection{Outline of the paper}

The structure of the paper is outlined as follows. 
In Section~\ref{section:approximation_formulas} we recall the (Dyadic) Approximation Formula from \cite{Hughes_Vinogradov} which decomposes our averages into a main term and an error term. We improve on the bounds for the error term of a single average. 
This improvement of the error term will be used in Section~\ref{section:main_theorem}. 
In Section~\ref{section:main_term}, we prove Theorem~\ref{theorem:main_term} which says that maximal function for the main term is restricted weak-type at the endpoint $\frac{\dimension}{\dimension-\degree}$. The proof is very similar to that in \cite{Ionescu_spherical}; as such, we assume the reader's familiarity with the Magyar--Stein--Wainger tranference principle and Bourgain's lemma in \cite{Ionescu_spherical}. 
In Section~\ref{section:main_theorem}, we prove Theorems~\ref{theorem:main_theorem} and \ref{theorem:low_density_improvement_by_mean_values}. 
We conclude the paper with Section~\ref{section:connection_to_Vinogradov} by connecting our Mean Value Hypothesis with Vinogradov's Mean Value Conjectures and recent work of T. Wooley. 
%

\subsection{Notations}

We use the same notations outlined in the previous paper \cite{Hughes_Vinogradov}. 
We recall these here for the reader's convenience. 
Here and throughout, $\eof{t}$ will denote the character $e^{2 \pi i t}$ for $t \in \R, \Zmod{\modulus}$ or $\T$. The torus $\T^\dimension := (\R/\Z)^\dimension$ is identified with the cube $[-1/2,1/2]^\dimension \subset \R^\dimension$. 
For two functions $f, g$, $f \lesssim g$ if $\absolutevalueof{f(x)} \leq C \absolutevalueof{g(x)}$ for some constant $C>0$. 
$f$ and $g$ are comparable $f \eqsim g$ if $f \lesssim g$ and $g \lesssim f$. 
All implicit constants throughout the paper may depend on dimension $\dimension$ and degree $\degree$. 
We will often identify $\Zmod{\modulus}$ with the set $\inbraces{1, \dots, \modulus}$, and $\unitsmod{\modulus}$, the group of units in $\Zmod{\modulus}$, will also be regarded as a subset of $\inbraces{1, \dots, \modulus}$ . 
For a set \(X\), we denote its indicator function by \(\indicator{X}\).

There are three Fourier transforms floating around. 
To distinguish these, if $f: \R^\dimension \to \C$, then we define its Fourier transform by $\contFT{\fxn}(\toruspoint) := \int_{\R^\dimension} f(x) e(x \cdot \toruspoint) dx$ for $\toruspoint \in \R^\dimension$; 
if $f: \T^\dimension \to \C$, then we define its Fourier transform by $\torusFT{\fxn}(\latticepoint) := \int_{\T^\dimension} f(x) e(-\latticepoint \cdot x) dx$ for $\latticepoint \in \Z^\dimension$; 
and if $f:\Z^\dimension \to \C$, then we define its Fourier transform by $\latticeFT{\fxn}(\toruspoint) := \sum_{\latticepoint \in \Z^\dimension} f(\latticepoint) e(n \cdot \toruspoint)$ for $\toruspoint \in \T^\dimension$.

\section*{Acknowledgements}
The author would like to thank Elias Stein for introducing him to the subject and for numerous conversations regarding the problems. 
The author thanks Lillian Pierce, Roger Heath-Brown, Tim Browning and Trevor Wooley for enlightening discussions related to the circle method. 
The author thanks Jim Wright for pointing out \cite{Steckin} and \cite{Duoandikoetxea_Vargas}, and Tony Carbery for a careful reading of the paper.

\newpage 
\section{The Approximation Formulas}\label{section:approximation_formulas}

A crucial insight of Magyar--Stein--Wainger - see in \cite{MSW_spherical} - in their proof of the boundedness of the discrete spherical maximal function is their approximation formula. Magyar generalized their approximation formula in \cite{Magyar_ergodic} for a class of forms including the $\degree$-spheres here. The author built on Magyar's work and sharpened his result for \(\degree\)-spheres in \cite{Hughes_Vinogradov} using a variant of Hypothesis \(\hypothesis{\powerloss}\) that included a log-loss. 
%
In this section we summarize the decomposition of the \(\degree\)-spherical measure and its bounds from the circle method approximation. 
In particular we recall the \emph{(Dyadic) Approximation Formula} from \cite{Hughes_Vinogradov} which will be used in Theorem~\ref{theorem:main_theorem}. Subsequently we improve the Dyadic Approximation Formula for a single average in the \emph{Single Approximation Formula} below. 
Both Approximation Formulas rely on bounds for exponential sums and oscillatory integrals. 
The Dyadic Approximation Formula makes use of our Sup Hypothesis \(\hypothesis{\powerloss}\) while the Single Approximation Formula makes use of the Sup Hypothesis and our Mean Value Hypothesis \(\meanvaluehypothesis{\dimension}\). 
The necessary bounds for the Fourier transform of the continuous $\degree$-spherical surface measures are significantly better than the analogous bounds for exponential sums. Since these bounds are implicit in the Approximation Formula, we do not recall them here; instead refer the vigilant reader to Section~3 of \cite{Hughes_Vinogradov}. 
Throughout the entire paper we assume that the dimension \(\dimension\) is sufficiently large so that 
\(
\numberoflatticepoints(\radius) 
\eqsim 
\radius^{\dimension-\degree} 
\) 
and renormalize the averages $\avgop$ as 
\[
\averagefxn(\latticepoint) 
= 
\frac{1}{c_{\dimension, \degree} \cdot \radius^{\dimension-\degree}} \sum_{n \in \arithmetichigherordersphere} \fxn(\latticepoint-n) 
\] 
where 
\(c_{\dimension, \degree} := \frac{\Gamma(1+1/\degree)^\dimension}{\Gamma(\dimension/\degree)}\) is the volume of the Gelfand--Leray form on \(\conthigherordersphere\).
For \(\toruspoint \in \T^\dimension\), the multiplier $\latticeFT{\avgop}(\toruspoint)$ for the convolution operator \(\avgop\) is given by \((c_{\dimension, \degree} \cdot \radius^{\dimension-\degree})^{-1} \cdot a_{\radius}(\toruspoint)\) where 
\[
a_{\radius}(\toruspoint) 
:=  \sum_{\latticepoint \in \arithmetichigherordersphere} \eof{\latticepoint \cdot \toruspoint} 
. 
\] 
Note that \(a_{\radius}(\toruspoint)\) is the Fourier transform of the characteristic function of the set of integer points on the $\degree$-sphere $\arithmeticsphere{\radius}$. 
For \(\radius \in \acceptableradii\), we have 
\begin{align*}
a_{\radius}(\toruspoint) 
& = 
\int_{\T} \sum_{||\latticepoint||_\infty \leq \radius} \eof{(\absolutevalueof{\latticepoint}^\degree-\radius^\degree)t + \latticepoint \cdot \toruspoint} \, dt 
\\ 
& = 
\int_{\T} \eof{-\radius^\degree t} \prod_{i=1}^{\dimension} \left( \sum_{|\latticepoint_i| \leq \radius} \eof{\absolutevalueof{\latticepoint_i}^\degree t + \latticepoint_i \toruspoint_i} \right) \, dt 
\end{align*}
where the first sum is over integer points in a cube of side-length \(2\radius\) centered at the origin and the second line follows from the tensor product nature of the exponential sum in the first line.
\footnote{We choose to sum over a cube rather than a ball as in \cite{Hughes_Vinogradov} so that we can exploit this tensor product structure in the proof of the Approximation Lemma}

The torus \(\T\), commonly identified with the interval \([0,1]\) via the character \(e(x) := e^{2\pi i x}\), decomposes into a disjoint union of major and minor arcs, commonly identified as collections of intervals in \([0,1]\). 
This decomposes $a_{\radius}$: 
\[
a_{\radius}(\toruspoint) 
= 
a^{Major}_{\radius}(\toruspoint) + a^{minor}_{\radius}(\toruspoint)
\]
where 
\begin{align*}
& a^{Major}_{\radius}(\toruspoint) 
:= 
\int_{\majorarcs} \sum_{||\latticepoint||_\infty \leq \dyadicradius} \eof{(\absolutevalueof{\latticepoint}^\degree-\radius^\degree)t + \latticepoint \cdot \toruspoint} \, dt 
\\ 
& a^{minor }_{\radius}(\toruspoint) 
:= 
\int_{\minorarcs} \sum_{||\latticepoint||_\infty \leq \dyadicradius} \eof{(\absolutevalueof{\latticepoint}^\degree-\radius^\degree)t + \latticepoint \cdot \toruspoint} \, dt 
. 
\end{align*}
Let $\avgop^{Major}$ and $\avgop^{minor}$ denote their respective normalized convolution operators. 
These multipliers are normalized so that 
\begin{align*}
& 
\latticeFT{\avgop^{Major}} 
= 
(c_{\dimension, \degree} \cdot \radius^{\dimension-\degree})^{-1} \cdot a^{Major}_{\radius} 
\\ 
& 
\latticeFT{\avgop^{minor}} 
= 
(c_{\dimension, \degree} \cdot \radius^{\dimension-\degree})^{-1} \cdot a^{minor}_{\radius}. 
\end{align*}
The multiplier corresponding to the major arcs, \(a^{Major}_{\radius}\) is then approximated by \(\radius^{\dimension-\degree} \cdot \HLmultiplier\).\footnote{
  There are three intermediate steps in the approximation. 
  The first defines multipliers \(b^{\unit/\modulus}_\radius\) to approximate each multiplier \(a^{\unit/\modulus}_{\radius}\) that compose the major arcs. 
  The second step approximates \(b^{\unit/\modulus}_\radius\) by \(c^{\unit/\modulus}_\radius := \radius^{\dimension-\degree} \HLmultiplier\). The final step `completes the singular series'. We refer the reader to \cite{Hughes_Vinogradov} for more details. 
}
Altogether, the averages decompose as 
\begin{equation*}
\avgop 
= 
\HLop + (\avgop^{Major} - \HLop) + \avgop^{minor} 
\end{equation*}
for each \(\radius \in \acceptableradii\). 
Our main focus of this section is the error term. 
In section~\ref{section:main_term}, we will describe the structure of the main term \(\HLop\) and bounds for it.

We see that the error term 
\[
\error 
:= 
\avgop - \HLop 
= 
(\avgop^{Major} - \HLop) + \avgop^{minor} 
\]
naturally composes of two pieces: \(\avgop^{Major} - \HLop\) and \(\avgop^{minor}\). 
The Dyadic Major Arc Aproximation Lemma (Section~7 of \cite{Hughes_Vinogradov}) reveals that we have the following bounds for the major arc piece of our error term:
\begin{equation}\label{equation:dyadic_major_arc_approximation}
\lpnorm{2}{\disup{\dyadicradius}{\radius}{\absolutevalueof{ \avgop^{Major}\fxn - \HLop\fxn} } } 
\lesssim 
\dyadicradius^{\degree+2-\frac{\dimension}{\degree}} \lpnorm{2}{\fxn} 
\end{equation}
\begin{remark}
The Dyadic Major Arc Aproximation Lemma of Section~7 of \cite{Hughes_Vinogradov} is actually stated with a a log-loss; that is, \(\lesssim\) is replaced with \(\lessapprox\) in  \eqref{equation:dyadic_major_arc_approximation}. 
However, this may be simply removed by replacing Hua's bound for Gauss sums in the proofs of Lemmas~7.1 and 7.2 in \cite{Hughes_Vinogradov} with Steckin's estimate. 
See section~\ref{section:main_term} for an example of this. 
We do not go into further details here. 
\end{remark}
We do not improve \eqref{equation:dyadic_major_arc_approximation} in this paper as this is not our goal here. Instead our aim is to understand and improve bounds for the minor arc piece \(\avgop^{minor}\). 

\subsection{The Approximation Formula: dyadic version}
\label{section:dyadic_approximation_lemma}

The analysis for the minor arc piece of our error term, \(\avgop^{minor}\) relies on \( \hypothesis{\powerloss} \) and proceeds along a different argument than our major arc piece. 
We handle the minor arc error term by Lemma~6.2 from \cite{Hughes_Vinogradov}: 
if $\hypothesis{\powerloss}$ is true for some \(\powerloss \in (0,1) \), then 
\begin{equation}\label{equation:dyadic_minor_arc_estimate}
\lpnorm{2}{\disup{\dyadicradius}{\radius}{\absolutevalueof{\avgop^{minor}\fxn} } } 
\lesssim 
\dyadicradius^{\degree - \dimension \powerloss} \lpnorm{2}{\fxn} 
. 
\end{equation}
\noindent
Combining \eqref{equation:dyadic_major_arc_approximation} and \eqref{equation:dyadic_minor_arc_estimate}, we obtain the Approximation Formula in \cite{Hughes_Vinogradov}: 
\begin{approximationlemma}[dyadic version]
If hypothesis $\hypothesis{\powerloss}$ is true for some $0 < \powerloss < 1$, then for $\dimension>\max{\{ \degree(\degree+2), \degree/\powerloss\}}$ and $\toruspoint \in \T^\dimension$,  
\begin{equation}\label{approximation_formula}
\arithmeticFT{\arithmetichigherorderspheremeasure}(\toruspoint) 
= \HLmultiplier(\toruspoint) + \arithmeticFT{\error}(\toruspoint) 
. 
\end{equation}
The error term $\arithmeticFT{\error}$ is a multiplier term with convolution operator $\error$ satisfying 
\begin{equation}\label{eq:error_bound}
\lpnorm{2}{\disup{\dyadicradius}{\radius}{|\error \fxn|}}
\lesssim \dyadicradius^{-\dyadicsavings} \lpnorm{2}{\fxn}
\end{equation}
where 
\(
\dyadicsavings 
:= 
\min \{ \dimension \powerloss - \degree, \frac{\dimension}{\degree}-(\degree+2) \} 
> 0
. 
\)
\end{approximationlemma}

Our goal in this section is to improve \eqref{equation:dyadic_minor_arc_estimate} for a single average. 
This will allow us to improve our range of boundedness for maximal operators over sparse sequences. 

\subsection{The Approximation Formula for a single average} 

Our adaptation of the Approximation Formula in \cite{Hughes_Vinogradov} for a single average is motivated by the method of directly exploiting the decay of the Fourier transform of the continuous spherical measure to study the continuous lacunary spherical maximal function versus the loss of a derivative in studying the full spherical maximal function. 
As in \cite{Hughes_Vinogradov} we are interested in the precise exponential sum bounds necessary in proving our theorems. 
To this effect we introduced our \emph{Mean Value Hypothesis} \(\meanvaluehypothesis{\dimension}\); see the introduction for its definition. 
%
%
With this in mind, we now have the following \emph{single average \(\ell^2\) inequality}. The reader may wish to compare this to the the \emph{Main \(\ell^2\) inequality} of  \cite{Hughes_Vinogradov} 
which is from the proof of (6.4) on page 204 of \cite{MSW_spherical} or the bottom of page 936 in \cite{Magyar_ergodic}. 
\begin{lemma}\label{lemma:single_ell_2_inequality}
If $T_r$ is an operator with multiplier 
\[
\latticeFT{T_r}(\xi) 
= 
\beta_r(\xi) 
:= 
\int_I \alpha(t,\xi) \eof{-t \radius^\degree} \; dt 
\] 
where 
\(
\alpha(t,\xi) : I \times \T^\dimension \to \C
\) 
and \(I\) is an interval in \([0,1]\), then for any $1 \leq p \leq \infty$ 
\begin{equation}\label{equation:new_main_inequality}
\lpnorm{2}{T_r f} 
\leq \sup_{\xi \in \T^\dimension} \inparentheses{\int_I \absolutevalueof{\alpha(t,\xi)}^p \; dt}^{1/p} \lpnorm{2}{f} 
\end{equation}
with the standard modification at $p=\infty$. 
\end{lemma}
\begin{proof}
By Plancherel's theorem, we have the familiar bound 
%
\[
\lpnorm{2}{T_r f} 
\leq 
\sup_{\xi \in \T^\dimension} \absolutevalueof{\beta_r(\xi)} \cdot \lpnorm{2}{f} 
. 
\]
Applying Holder's inequality we bound the first factor: 
\begin{align*}
\absolutevalueof{\beta_r(\xi)} 
& \leq 
\int_I \absolutevalueof{\alpha(t,\xi)} \; dt 
\\ 
& \leq 
|I|^{1/p'} \inparentheses{\int_I \absolutevalueof{\alpha(t,\xi)}^p \; dt}^{1/p} 
\\ 
& \leq 
\inparentheses{\int_I \absolutevalueof{\alpha(t,\xi)}^p \; dt}^{1/p} 
\end{align*}
since \(|I| \leq 1\). 
%
\end{proof}

The bound \eqref{equation:new_main_inequality} works best for $p=1$ when we exploit the tensor product nature of our exponential sums. 
The following lemma allows us to ignore the possible cancellation arising from the linear phases. 
The method here is common in number theory and proceeds by using Plancherel's theorem to rewrite the $L^p$-norm as the number of solutions to a system of equations. 
Hu--Li in \cite{Hu_Li1}, \cite{Hu_Li2} and \cite{Hu_Li3} recently used this method to study related restriction problems. 
Since the proof of the following lemma is a standard technique in the circle method, we refer the reader to see Chapter~5, Section~5.1 of \cite{Vaughan}, in particular inequality (5.4), or Chapter~4, section~2 of \cite{Montgomery} for similar proofs. 
We define the following notation for the remainder of this section: if \(\xi \in \T\) and \(t \in I\) for some interval \(I\), let 
\[
\alpha_\radius(t,\xi) 
:= 
\sum_{|\latticepoint| \leq \radius} \eof{\absolutevalueof{\latticepoint}^\degree t + \latticepoint \toruspoint} 
. 
\]

\begin{lemma}\label{lemma:remove_linear_phases}
For \(\moment \in \N\) and any $\xi_i, t \in \T$ with \(i=1, \dots, 2\moment\), 
\begin{equation}
{\int_0^1 \prod_{i=1}^{2\moment} \absolutevalueof{\alpha_\radius(t,\xi_i)} \; dt} 
\leq {\int_0^1 \absolutevalueof{\alpha_\radius(t,0)}^{2\moment} \; dt} 
\end{equation}
\end{lemma}

We can now state and prove our Approximation Formula for a single average. 

\begin{approximationlemma}[single average version]
If hypothesis \(\hypothesis{\powerloss}\) is true for some \(\powerloss \in (0,1)\) and \(\meanvaluehypothesis{\moment}\) is true for some dimension \(\moment \in \N\), then for \(\dimension>\max{\{ 2\moment, \degree(\degree+2)\}}\) and $\toruspoint \in \T^\dimension$,  
\begin{equation}\label{approximation_formula}
\arithmeticFT{\arithmetichigherorderspheremeasure}(\toruspoint) 
= 
\HLmultiplier(\toruspoint) + \arithmeticFT{\error}(\toruspoint) 
. 
\end{equation}
The error term $\arithmeticFT{\error}$ is a multiplier with convolution operator $\error$ satisfying 
\begin{equation}\label{eq:error_bound}
\lpnorm{2}{\error \fxn}
\lesssim 
\dyadicradius^{-\singlesavings} \lpnorm{2}{\fxn}
\end{equation}
where 
\(
\singlesavings 
:= 
\min \{ (\dimension-2\moment)\powerloss , \degree+2-\frac{\dimension}{\degree} \}
\). 
\end{approximationlemma}

\begin{remark}
The Single Approximation Formula yields a power savings of $(\dimension-2\moment)\powerloss$ which is weaker than that of $\dimension \powerloss - \degree$, but true for a larger range of dimensions. 
The catch is that the Single Approximation Formula/Lemma only allows us to control a single average at a time rather than a dyadic range. 
However, this is useful for maximal functions over sparser averages, such as lacunary maximal functions. 
\end{remark}

\begin{proof}[of The Approximation Formula]
By \eqref{equation:dyadic_major_arc_approximation}, we only need to study the minor arcs and improve \eqref{equation:dyadic_minor_arc_estimate} when \(\dimension>2\moment\). 
We do so for an individual average by considering \(\lpnorm{2}{\avgop^{minor}\fxn}\) for each \(\radius \in \acceptableradii\). 
Plancherel's theorem reduces our goal to a uniform exponential sum estimate of 
\(
a^{minor}_{\radius}(\toruspoint) 
\) 
for \(\toruspoint = (\xi_1, \dots, \xi_\dimension) \in \T^\dimension\). 
The following is a typical approach for bounding minor arcs in Waring's problem; see for instance Chapter~3 in \cite{Davenport}.

%
We take $p=1$ in Lemma~\ref{lemma:single_ell_2_inequality} in order to exploit the tensor product nature of the exponential sums. 
Since \(\dimension > 2\moment\), 
\begin{align*}
\lpnorm{2}{a^{minor}_{\radius}} 
& \leq 
\sup_{\xi \in \T^\dimension} \inbraces{\inparentheses{\int_I \absolutevalueof{\prod_{i=1}^\dimension \alpha_\radius(t,\xi_i)} \; dt}} 
\\ 
& \leq 
\sup_{\xi \in \T^\dimension} \inbraces{\inbrackets{\sup_{t \in I} \inbraces{\prod_{i=2\moment+1}^{\dimension} \absolutevalueof{\alpha_\radius(t,\xi_i)}}} 
    \cdot \inparentheses{\int_I \prod_{i=1}^{2\moment} \absolutevalueof{\alpha_\radius(t,\xi_i)} \; dt}} 
\\ 
& \leq 
\sup_{\xi \in \T^\dimension} \inbraces{\inbrackets{\sup_{t \in I} \inbraces{\prod_{i=2\moment+1}^{\dimension} \absolutevalueof{\alpha_\radius(t,\xi_i)}}} 
    \cdot \inparentheses{\int_0^1 \prod_{i=1}^{2\moment} \absolutevalueof{\alpha_\radius(t,\xi_i)} \; dt}} 
\\ 
& \leq 
\sup_{\xi \in \T^\dimension} \inbraces{\inbrackets{\sup_{t \in I} \inbraces{\prod_{i=2\moment+1}^{\dimension} \absolutevalueof{\alpha_\radius(t,\xi_i)}}} 
    \cdot \inparentheses{\int_0^1 \prod_{i=1}^{2\moment} \absolutevalueof{\alpha_\radius(t,0)} \; dt}} 
\end{align*}
where the last line above follows from Lemma~\ref{lemma:remove_linear_phases}. 

Plancherel's Theorem implies that for any \(s \in \N\) 
\begin{equation}\label{eq:solutions_to_equations}
\int_0^1 \absolutevalueof{\alpha_\radius(t,0)}^{2\moment} \; dt 
= \# \setof{m_1, \dots, m_\moment ; n_1, \dots, n_\moment \in [r] : \sum_{i=1}^\moment |n_i|^\degree = \sum_{i=1}^\moment |m_i|^\degree} 
. 
\end{equation}
\noindent 
Our mean value hypothesis implies by \eqref{eq:solutions_to_equations} that 
\begin{equation*}
\int_0^1 \absolutevalueof{\alpha_\radius(t,0)}^{2\moment} \; dt 
\lesssim 
\radius^{2\moment-\degree} 
. 
\end{equation*}
Therefore, 
%
\begin{equation}\label{equation:single_minor_arc_estimate}
\lpnorm{2}{A_\radius^{minor}} 
\lesssim 
  \radius^{\degree-\dimension} 
  \cdot \radius^{2\moment-\degree} 
  \cdot \radius^{(\dimension-2\moment)(1-\powerloss)} 
= 
\radius^{-(\dimension-2\moment)\powerloss} 
. 
\end{equation}
%
\end{proof}

%
%

\section{The main term is restricted weak-type at the endpoint for $\dimension > 2 \degree$}\label{section:main_term}

We delve further into the Approximation Formula by refining estimates for the main term. 
The multiplier for the main term decomposes into a sum of operators 
\begin{equation}
\HLmultiplier 
= 
\sum_{\modulus=1}^\infty \sum_{\unit \in \unitsmod{\modulus}} \eof{ \frac{\unit \radius^\degree}{\modulus} } \fracHLmultiplier 
\end{equation} 
where we have the multipliers 
\begin{equation} 
\fracHLmultiplier(\toruspoint) 
:= \sum_{\latticepoint \in \Z^\dimension} \twistedGausssum{\latticepoint} \Psi(\modulus \toruspoint - \latticepoint) \contFT{\conthigherorderspheremeasure}(\toruspoint - \latticepoint/\modulus) 
\end{equation}
and  
\begin{itemize}
\item 
$\Psi$ is a smooth bump function supported in $[-1/4, 1/4]^\dimension$ and equal to 1 in $[-1/8, 1/8]^\dimension$, 
\item 
for \(\latticepoint \in \Z^\dimension, \modulus \in \N\) and \(\unit \in \unitsmod{\modulus}\) 
\[
\twistedGausssum{\latticepoint} 
:= \modulus^{-\dimension} \prod_{i=1}^\dimension \inbrackets{\sum_{b_i \in \Zmod{\modulus}} \eof{\frac{a b_i^\degree + b_i \cdot \latticepoint_i}{\modulus}}} 
\] 
is a normalized Gauss sum, 
\item 
\(\conthigherorderspheremeasure\) is the Gelfand--Leray form on \(\conthigherordersphere := \{x \in \R^\dimension : \sum_{i=1}^{\dimension} |x_i|^\degree=r^\degree\}\) normalized to be a probability measure whose \(\R^\dimension\)-Fourier transform is denoted \( \contFT{\conthigherorderspheremeasure} \) . 
\end{itemize}

\begin{remark}
The Approximation Formula is a generalization of the asymptotic formula in Waring's problem. As such, the main term \(\HLop\) connects analysis on $\Z^\dimension$ with with the analysis of  $(\Zmod{\modulus})^\dimension$ and $\R^\dimension$. In particular, we will compare $\arithmetichigherordersphere := \{\latticepoint \in \Z^\dimension : \sum_{i=1}^{\dimension} |\latticepoint_i|^\degree=r^\degree\}$ with its projections \emph{mod $\modulus$} and its embedding in $\conthigherordersphere := \{x \in \R^\dimension : \sum_{i=1}^{\dimension} |x_i|^\degree=r^\degree\}$ through their respective measures. 
The measure for $\arithmetichigherordersphere$ is the probability measure $\arithmetichigherorderspheremeasure = \frac{1}{\numberoflatticepoints(\radius)} \indicator{\arithmetichigherordersphere}$ where $\indicator{\arithmetichigherordersphere}$ is the characteristic function of $\arithmetichigherordersphere$, and the measure for  $\conthigherordersphere$ is $\conthigherorderspheremeasure$ given by the Gelfand--Leray form on $\conthigherordersphere$. 
More precisely, we will approximate the $\Z^\dimension$-Fourier transform of the measure $\arithmetichigherorderspheremeasure$ by the $\R^\dimension$-Fourier transform of $\conthigherorderspheremeasure$ and its projected measure in $(\Zmod{\modulus})^\dimension$ which are given by the normalized Gauss sums \( \twistedGausssum{\latticepoint} \).  
\end{remark}

We have the following lemma for the maximal function of the \(\HLop\). 
The proof relies on Hua's bound for Gauss sums, the Bruna--Nagel--Wainger bounds for Fourier transforms of surface measures and Rubio de Francia's maximal theorem for surface measures. 
For its proof, see the proof of Lemma 8.1 in \cite{Hughes_Vinogradov}. 
\begin{lemma}\label{lemma:maxHLop_bound}
If $\dimension > 2\degree+1$ and $p > \frac{\dimension}{\dimension-\degree}$, then 
\begin{equation}\label{maxHLop_bound}
\lpnorm{p}{\sup_{\radius \in \acceptableradii} \absolutevalueof{\HLop \fxn}} 
\lesssim \lpnorm{p}{\fxn} 
. 
\end{equation}
\end{lemma}
The purpose of this section is to refine Lemma~\ref{lemma:maxHLop_bound} to a restricted weak-type bound at the endpoint when the dimension is sufficiently large. 
\begin{theorem}\label{theorem:main_term}
Assume that the degree $\degree \geq 3$. 
If $\dimension \geq 2 \degree+1$, then the maximal function 
\(\sup_{\radius \in \acceptableradii} \absolutevalueof{\HLop \fxn}\) 
is restricted weak-type \((\frac{\dimension}{\dimension-\degree}, \frac{\dimension}{\dimension-\degree})\). 
\end{theorem}

Our approach is the same as \cite{Ionescu_spherical}. 
In particular, Theorem~\ref{theorem:main_term} is easily deduced from the following \emph{decomposition lemma}. 
\begin{lemma}[Decomposition lemma for the main term]\label{lemma:main_term_decomposition}
For any fixed $\bigmodulus \in \N$ we can decompose each $\degree$-spherical average of $\radius \in \acceptableradii$ into the sum of 2 linear operators: 
\begin{equation}
\HLop \fxn(x) 
= \HLop^{high} \fxn(x) 
  + \HLop^{low} \fxn(x) 
. 
\end{equation}
such that if $\dimension>2\degree$, then the following bounds are satisfied: 
\begin{align}
& \lpnorm{2}{\sup_{\radius \in \acceptableradii} \absolutevalueof{\HLop^{high} \fxn}} 
\lesssim \bigmodulus^{2-\frac{\dimension}{\degree}} \lpnorm{2}{\fxn} 
\label{eq:high_freq_estimate} 
& \text{(High frequency estimate)} 
\\ & \lpnorm{1,\infty}{\sup_{\radius \in \acceptableradii} \absolutevalueof{\HLop^{low} \fxn}} 
\lesssim \bigmodulus^2 \lpnorm{1}{\fxn} 
\label{eq:low_freq_estimate} 
& \text{(Low frequency estimate)} 
. 
\end{align}
\end{lemma}

First we deduce Theorem~\ref{theorem:main_term} from Lemma~\ref{lemma:main_term_decomposition}. 
Subsequently, we prove Lemma~\ref{lemma:main_term_decomposition}.

\subsection{Deduction of Theorem~\ref{theorem:main_term} from Lemma~\ref{lemma:main_term_decomposition}}
We want to show that 
\[
\sup_{\altitude > 0} \altitude^{\frac{\dimension}{\dimension-\degree}} \absolutevalueof{\inbraces{\maxfxn > \altitude}} 
\lesssim \lpnorm{\frac{\dimension}{\dimension-\degree}}{\fxn}^{\frac{\dimension}{\dimension-\degree}}
\] 
for any function in $\ell^{\frac{\dimension}{\dimension-\degree}}(\Z^\dimension)$; instead we prove this for the characteristic function of any subset $\latticesubset$ in $\Z^\dimension$. 
Let $\indicator{\latticesubset}$ denote the characteristic function of a set $\latticesubset$.
To get this restricted weak-type bound, we will choose $\bigmodulus$ depending on an altitude $\altitude>0$.
Suppose for a moment that we can choose our parameters so that
\begin{align}
\bigmodulus^2 
& \lesssim \bigradius^{\degree} \label{eq:ell1_condition} \\
\bigmodulus^{2-\frac{\dimension}{\degree}} 
& \lesssim \bigradius^{\degree - \dimension/2} \label{eq:ell2_condition}
. 
\end{align}
Then we have the bounds 
\begin{align}
& \lpnorm{1,\infty}{\sup_{\radius \in \acceptableradii} \absolutevalueof{\HLop^{high} \fxn}} 
\lesssim \bigradius^{\degree} \lpnorm{1}{\fxn} \label{eq:ell1_bound} 
\\
& \lpnorm{2}{\sup_{\radius \in \acceptableradii} \absolutevalueof{\HLop^{low} \fxn}} 
\lesssim \bigradius^{\degree - \dimension/2} \lpnorm{2}{\fxn} \label{eq:ell2_bound} 
. 
\end{align}
Applying \eqref{eq:ell1_bound} and \eqref{eq:ell2_bound} to $\indicator{\latticesubset}$, the characteristic function of the set $\latticesubset \in \Z^\dimension$, we find 
\begin{align*}
\absolutevalueof{\inbraces{\maxop \indicator{\latticesubset} > 2 \altitude}} 
& \leq \absolutevalueof{\inbraces{\sup_{\radius \in \acceptableradii} \absolutevalueof{\HLop^{high} \indicator{\latticesubset} > \altitude}}} 
    + \absolutevalueof{\inbraces{\sup_{\radius \in \acceptableradii} \absolutevalueof{\HLop^{low} \indicator{\latticesubset}} > \altitude}} \\
& \lesssim \altitude^{-1} \bigradius^{\degree} \absolutevalueof{\latticesubset} + \altitude^{-2} \bigradius^{2(\degree - \dimension/2)} \absolutevalueof{\latticesubset} \\
& = \inparentheses{\altitude^{-1} \bigradius^{\degree} + \altitude^{-2} \bigradius^{2\degree - \dimension} } \absolutevalueof{\latticesubset}
.
\end{align*}
We balance the two terms on the right hand side by choosing $\bigradius = \altitude^{\frac{1}{\degree - \dimension}}$ so that $\altitude^{-1} \bigradius^{\degree} \eqsim \altitude^{-2} \bigradius^{2\degree - \dimension}$. Plugging this into the above, we find that 
\[
\absolutevalueof{\inbraces{\maxop \indicator{\latticesubset} > 2 \altitude}}
\lesssim \altitude^{-1-\frac{\degree}{\dimension - \degree}} \absolutevalueof{\latticesubset}
= \altitude^{-\frac{\dimension}{\dimension - \degree}} \absolutevalueof{\latticesubset} 
\]
which is the weak-type bound we seek. 

It is easy to verify that \eqref{eq:ell1_condition} and \eqref{eq:ell2_condition} hold for $\bigmodulus \eqsim \bigradius^{\degree/2}$ provided that $\dimension > 2 \degree$. 
%
%

\subsection{Proof of the decomposition lemma for the main term (Lemma~\ref{lemma:main_term_decomposition})}
In this section we outline the proof of Lemma~\ref{lemma:main_term_decomposition}. The details are similar to the proofs of estimates (2.9) and (2.10) in \cite{Ionescu_spherical} making the necessary modifications to higher degrees like those in \cite{Hughes_Vinogradov} from \cite{MSW_spherical}. One important point is our use of Ste\v{c}kin's estimate \eqref{eq:Steckin_estimate} for the Gauss sums $\twistedGausssum{\latticepoint}$ rather than Hua's estimate for them (see \emph{Hua's bound} in \cite{Hughes_Vinogradov}); otherwise, we cannot reach the endpoint \(p = \frac{\dimension}{\dimension-\degree}\). 
We recall Ste\v{c}kin's estimate now.  
\begin{Steckin}[\cite{Steckin}]
If $(\unit, \modulus)=1$, then 
\begin{equation}\label{eq:Steckin_estimate}
|\twistedGausssum{\latticepoint}| 
\lesssim \modulus^{-\frac{\dimension}{\degree}} 
\end{equation}
uniformly for $\latticepoint \in \Z^\dimension$. 
\end{Steckin}

For the operators $\fracHLop$, we have high frequency in two aspects: the modulus $\modulus$ and the continuous aspect. 
%
We decompose $\fracHLop$ into \emph{continuous-high} and \emph{continuous-low frequency multipliers}. 
Let $0<\bigfrequencydilation<1$, and for each $\radius \in \acceptableradii$ define the multipliers 
\begin{align}
\lowfracHLmultiplier(\toruspoint) 
&:= \sum_{\latticepoint \in \Z^\dimension} \twistedGausssum{\latticepoint} \dilateby{\modulus}\Psi(\toruspoint - \latticepoint/\modulus) \contFT{d\surfacemeasure_\radius}(\toruspoint - \latticepoint/\modulus) \cdot \dilateby{2 \modulus \bigfrequencydilation \radius} \Psi(\toruspoint - \latticepoint/\modulus) 
\label{eq:low_op}
\\
\highfracHLmultiplier(\toruspoint) 
&:= \sum_{\latticepoint \in \Z^\dimension} \twistedGausssum{\latticepoint} \dilateby{\modulus}\Psi(\toruspoint - \latticepoint/\modulus) \contFT{d\surfacemeasure_\radius}(\toruspoint - \latticepoint/\modulus) \cdot \inbrackets{1-\dilateby{2 \modulus \bigfrequencydilation \radius} \Psi(\toruspoint - \latticepoint/\modulus)} 
\label{eq:cont-high_op}
. 
\end{align}
Here $\dilateby{t} \Psi(\toruspoint) = \Psi(t \toruspoint)$ for $t>0$. 
Fix $\bigmodulus>0$. 
Let 
\begin{align}
\lowHLop 
& = \sum_{\modulus < \bigmodulus} \sum_{\unit \in \unitsmod{\modulus}} \lowfracHLop
\label{eq:low_op}
\\
\highHLop 
& = \inparentheses{\sum_{\modulus < \bigmodulus} \sum_{\unit \in \unitsmod{\modulus}} \highfracHLop} 
    + \inparentheses{\sum_{\modulus \geq \bigmodulus} \sum_{\unit \in \unitsmod{\modulus}} \fracHLop} 
\label{eq:cont-high_op}
. 
\end{align}
so that we have $\HLop = \lowHLop + \highHLop$. 
At the moment, our decomposition depends on $\bigfrequencydilation$ and $\bigmodulus$, but we will soon choose $\bigfrequencydilation = \bigmodulus^{-1}$. 
%

%

Replacing Hua's bound by Ste\v{c}kin's estimate \eqref{eq:Steckin_estimate} in the proof of Lemma~7.3 in \cite{Hughes_Vinogradov}, we have the following improvement to Lemma~7.3 in \cite{Hughes_Vinogradov}. 
\begin{lemma}\label{lemma:Steckin_for_multipliers}
If $\dimension > \frac{\degree}{2}+1$, then for all moduli $\modulus \in \N$ and $\unit \in \unitsmod{\modulus}$, we have 
\begin{equation}\label{eq:Steckin_for_multipliers}
\lpnorm{2}{\sup_{\radius \in \acceptableradii} \absolutevalueof{\fracHLop \fxn}} 
\lesssim \modulus^{-\frac{\dimension}{\degree}} \lpnorm{2}{\fxn} 
. 
\end{equation}
\end{lemma}
%

We will apply Lemma~\ref{lemma:Steckin_for_multipliers} for all large moduli, $\modulus \geq \bigmodulus$, but this lemma is insufficient for small moduli, $\modulus < \bigmodulus$. However, we obtain a good bound for the continuous-high frequency multipliers. 
As in \cite{Ionescu_spherical}, we use the Magyar--Stein--Wainger transference principle in \cite{MSW_spherical} and Lemma~?? from \cite{Bourgain??} to show that the maximal function for the high frequency part, $\highfracHLop$ has good $\ell^2$ estimates due to the decay of the Fourier transform of the (continuous) surface measure $\contFT{d\surfacemeasure_\radius}$. 
\begin{lemma}\label{lemma:continuous_high_frequency_bound}
If $\dimension>\frac{\degree}{2}+1$ and $0<\bigfrequencydilation<1$, then 
\begin{equation}\label{eq:continuous_high_frequency_bound}
\lpnorm{2}{\sup_{\radius \in \acceptableradii} \absolutevalueof{\highfracHLop \fxn}} 
\lesssim {\modulus}^{-\frac{\dimension}{\degree}} \cdot \inparentheses{\modulus \bigfrequencydilation}^{\frac{\dimension-1}{\degree}-\frac12} \lpnorm{2}{\fxn} 
. 
\end{equation} 
\end{lemma}
\begin{proof}
We apply the Magyar--Stein--Wainger separation trick to separate out the arithmetic and analytic parts: $\highfracHLop = T^{\unit/\modulus}_\radius \circ S^{\unit/\modulus}$, defined by the multipliers  
\begin{align}
\arithmeticFT{S^{\unit/\modulus}}(\toruspoint) 
& = \sum_{\latticepoint \in \Z^\dimension} \twistedGausssum{\latticepoint} \dilateby{\modulus}\Psi'(\toruspoint - \latticepoint/\modulus) 
\\
\arithmeticFT{T^{\unit/\modulus}_\radius}(\toruspoint) 
& = \sum_{\latticepoint \in \Z^\dimension} \dilateby{\modulus/2}\Psi(\toruspoint - \latticepoint/\modulus) \cdot \inbrackets{1-\dilateby{2 \modulus \bigfrequencydilation \radius} \Psi(\toruspoint - \latticepoint/\modulus)} \contFT{d\surfacemeasure_\radius}(\toruspoint - \latticepoint/\modulus) 
. 
\end{align}
Note that $S^{\unit/\modulus}$ does not depend on the radius $\radius \in \acceptableradii$; this implies 
\begin{equation}
\elltwooperatornorm{\sup_{\radius \in \acceptableradii} \absolutevalueof{\highfracHLop}} 
\leq \elltwooperatornorm{\sup_{\radius \in \acceptableradii} \absolutevalueof{T^{\unit/\modulus}_\radius}} 
    \cdot \elltwooperatornorm{{S^{\unit/\modulus}}} 
. 
\end{equation}

For the arithmetic part, Ste{\v{c}}kin's estimate \eqref{eq:Steckin_estimate} and Proposition~2.2 of \cite{MSW_spherical} imply that 
\begin{equation}
\elltwooperatornorm{{S^{\unit/\modulus}}} 
\lesssim {\modulus}^{-\frac{\dimension}{\degree}} 
. 
\end{equation}

For the analytic part, we apply the Magyar--Stein--Wainger transference principle (Proposition~2.1 and Corollary~2.1 in \cite{MSW_spherical}). Define the operator $U_\radius$ by the multiplier 
\[
\contFT{U_\radius}(\toruspoint)
:= \dilateby{\modulus/2}\Psi(\toruspoint) \inbrackets{1-\dilateby{2 \modulus \bigfrequencydilation \radius} \Psi(\toruspoint)} \contFT{d\surfacemeasure_\radius}(\toruspoint) 
\] 
which is considered as a multiplier on $\R^\dimension$. 
Then the Magyar--Stein--Wainger transference principle implies 
\begin{equation}
\ellpoperatornorm{\sup_{\radius \in \acceptableradii} \absolutevalueof{T^{\unit/\modulus,low}_\radius}} 
\lesssim \operatornorm{L^p(\R^\dimension) \to L^p(\R^\dimension)}{\sup_{\radius \in \acceptableradii} \absolutevalueof{U_\radius}} 
\end{equation}
for $1 \leq p < \infty$. 

We normalize $\contFT{U_\radius}$ so that it does not depend on $\radius$ by observing $\contFT{U_\radius}(\toruspoint) = \contFT{U_1}(\radius \toruspoint) = \dilateby{\radius} \contFT{U_1}(\toruspoint)$.  
%
We are in position to apply Bourgain's lemma. Bourgain's lemma in \cite{Ionescu_spherical} tells us that  
\begin{equation}
\operatornorm{L^p(\R^\dimension) \to L^p(\R^\dimension)}{\sup_{\radius \in \acceptableradii} \absolutevalueof{U_\radius}} 
\lesssim \sum_{j \in \Z} \alpha_j^{1/2} (\alpha_j^{1/2} + \beta_j^{1/2}) 
\end{equation}
where $\alpha_j := \sup_{|\toruspoint| \eqsim 2^j} |\contFT{U_1}(\toruspoint)|$ and $\beta_j := \sup_{|\toruspoint| \eqsim 2^j} |\toruspoint \cdot \nabla\contFT{U_1}(\toruspoint)|$. 
%
Now we merely need to calculate $\alpha_j$ and $\beta_j$ for $j \in \Z$. The support condition implies that $\alpha_j$ and $\beta_j$ are 0 for $2^j \leq \inparentheses{8 \modulus \bigfrequencydilation}^{-1}$ so that we only need to consider $j$ such that $\inparentheses{8 \modulus \bigfrequencydilation}^{-1} \leq 2^{j}$. 
Otherwise, the Bruna--Nagel--Wainger bounds in \cite{BNW}, see also (1) and (2) in Section~3 of \cite{Hughes_Vinogradov}, yield
\begin{align*}
\alpha_j 
& \lesssim \inparentheses{1+2^j}^{\frac{1-\dimension}{\degree}} 
, 
\\
\beta_j 
& \lesssim 2^j \cdot\inparentheses{1+2^j}^{\frac{1-\dimension}{\degree}} 
. 
\end{align*}
Applying these bounds, we sum over $\inparentheses{8 \modulus\bigfrequencydilation}^{-1} \leq 2^j$ to conclude the lemma. 
%
\end{proof}

Each of the remaining low frequency parts, $\lowfracHLop$ is comparable to the discrete Hardy--Littlewood averages and thus its maximal function is comparable to the Hardy--Littlewood maximal function with a bound that depends on the modulus $\modulus$. 
\begin{proposition}\label{proposition:continuous_low_frequency_bound}
Let $\discreteHLmaxop$ be the discrete Hardy--Littlewood maximal function over cubes. 
If $\dimension \geq 2$, $\degree \geq 2$ and $0<\bigfrequencydilation<1$, then 
\begin{equation}\label{eq:continuous_low_frequency_comparison}
{\sup_{\radius \in \acceptableradii} \absolutevalueof{\lowfracHLop \fxn(x)}} 
\lesssim \inparentheses{\modulus \bigfrequencydilation}^{-1} \discreteHLmaxop \fxn(x) 
\end{equation}
for all $x \in \Z^\dimension$. 
\end{proposition}

\begin{proof}
Note that for $\radius \bigfrequencydilation \geq 1$, 
\[
\dilateby{\modulus}\Psi(\toruspoint - \latticepoint/\modulus) \cdot \dilateby{2 \modulus \bigfrequencydilation \radius} \Psi(\toruspoint - \latticepoint/\modulus) 
= \dilateby{2 \modulus \bigfrequencydilation \radius} \Psi(\toruspoint - \latticepoint/\modulus) 
\]
is a smooth function supported in $\inbrackets{-1/(8 \bigfrequencydilation \radius), 1/(8 \bigfrequencydilation \radius )}^{\dimension}$. 
This implies 
\[
\lowfracHLmultiplier(\toruspoint) 
= \sum_{\latticepoint \in \Z^\dimension} \twistedGausssum{\latticepoint} \dilateby{2 \modulus \bigfrequencydilation \radius} \Psi(\toruspoint - \latticepoint/\modulus) \contFT{d\surfacemeasure_\radius}(\toruspoint - \latticepoint/\modulus) 
\]
for $\radius \bigfrequencydilation \geq 1$. 

A straightforward computation, see page 1415 of \cite{Ionescu_spherical}, reveals that the kernel $\lowfracHLkernel$ of $\fracHLop$ is 
\begin{equation}\label{eq:kernel_identity}
\lowfracHLkernel(x) 
= \eof{-\frac{\absolutevalueof{x}^{\degree}}{\modulus}} \radius^{-\dimension} \contFT{\dilateby{2 \modulus \bigfrequencydilation} \Psi} \convolvedwith {d\surfacemeasure}(x/\radius)
\end{equation}
for each $x \in \Z^\dimension$. 
A standard argument shows 
\begin{equation}\label{eq:blurred_sphere_bound}
\contFT{\dilateby{t} \Psi} \convolvedwith {d\surfacemeasure}(x)
\lesssim_N t^{-1} (1+\absolutevalueof{x})^{-N} 
\end{equation}
for any $t>0$ and any $N \in \N$. 
Taking $t = 2 \modulus \bigfrequencydilation$ and $N = \dimension+1$, \eqref{eq:kernel_identity} and \eqref{eq:blurred_sphere_bound} imply
\[
\absolutevalueof{\lowfracHLkernel(x)} 
\lesssim \radius^{-\dimension} 
\inparentheses{\modulus \bigfrequencydilation}^{-1} (1+\absolutevalueof{x/\radius})^{-\dimension-1} 
. 
\]
This is an approximation to the identity which implies that 
\[
\sup_{\radius > 0} \absolutevalueof{\lowfracHLop \fxn} 
\lesssim \inparentheses{\modulus \bigfrequencydilation}^{-1} \cdot \discreteHLmaxop \fxn 
. 
\]
%
\end{proof}
%
%


\begin{proof}[of Lemma~\ref{lemma:main_term_decomposition}]
We choose $\bigfrequencydilation = \bigmodulus^{-1}$. 
By \eqref{eq:cont-high_op}, \eqref{eq:Steckin_for_multipliers} and \eqref{eq:continuous_high_frequency_bound}, we see that if $\dimension>\frac{\degree}{2}+1$ and $0<\bigfrequencydilation<1$, then 
\begin{align*}
\lpnorm{2}{\sup_{\radius \in \acceptableradii} \absolutevalueof{\highHLop \fxn}} 
& \lesssim 
    \inparentheses{
    \sum_{\modulus < \bigmodulus} \sum_{\unit \in \unitsmod{\modulus}} \lpnorm{2}{\sup_{\radius \in \acceptableradii} \absolutevalueof{\highHLop \fxn}} 
    }
\\
&   \hspace{5mm} + \inparentheses{
    \sum_{\modulus \geq \bigmodulus} \sum_{\unit \in \unitsmod{\modulus}} \lpnorm{2}{\sup_{\radius \in \acceptableradii} \absolutevalueof{\HLop \fxn}} 
    }
\\ 
& \lesssim \inparentheses{
    \sum_{\modulus < \bigmodulus} \sum_{\unit \in \unitsmod{\modulus}} {\modulus}^{-\frac{\dimension}{\degree}} \cdot \inparentheses{\modulus \bigfrequencydilation}^{\frac{\dimension-1}{\degree}-\frac12} 
    + \sum_{\modulus \geq \bigmodulus} \sum_{\unit \in \unitsmod{\modulus}} \modulus^{-\frac{\dimension}{\degree}} 
    } \lpnorm{2}{\fxn} 
\\ 
& \lesssim {\bigmodulus}^{2-\frac{\dimension}{\degree}} \lpnorm{2}{\fxn} 
. 
\end{align*} 
This is \eqref{eq:high_freq_estimate}. 

Similarly, \eqref{eq:continuous_low_frequency_comparison} shows that 
\begin{align*}
\sup_{\radius \in \acceptableradii} \absolutevalueof{\lowHLop \fxn} 
& \leq \sum_{\modulus < \bigmodulus} \sum_{\unit \in \unitsmod{\modulus}} \sup_{\radius \in \acceptableradii} \absolutevalueof{\lowHLop \fxn} 
\\ & \lesssim \sum_{\modulus < \bigmodulus} \sum_{\unit \in \unitsmod{\modulus}} \inparentheses{\modulus \bigfrequencydilation}^{-1} \discreteHLmaxop \fxn 
\\ & \leq {\bigmodulus}^2 \discreteHLmaxop \fxn 
. 
\end{align*} 
Therefore the Hardy--Littlewood maximal theorem implies that 
\[
\lpnorm{1, \infty}{\sup_{\radius \in \acceptableradii} \absolutevalueof{\lowHLop \fxn}} 
\lesssim {\bigmodulus}^2 \lpnorm{1, \infty}{\discreteHLmaxop \fxn} 
\lesssim {\bigmodulus}^2 \lpnorm{1}{\fxn} 
. 
\]
This is \eqref{eq:low_freq_estimate}. 
\end{proof}

\section{Proofs of Theorems~\ref{theorem:main_theorem} and \ref{theorem:low_density_improvement_by_mean_values}}\label{section:main_theorem}

Having completed the proof of Theorem~\ref{theorem:main_term}, we turn our attention to Theorem~\ref{theorem:main_theorem}. The proof of Theorem~\ref{theorem:main_theorem} will be similar to the proof of Theorem~\ref{theorem:main_term}, but simpler. The simplicity is due in part to our notion of density-parameter and our use of the union bound in Proposition~\ref{proposition:density_ell1_bound}. 

\begin{definition}
A subsequence $\radii$ in $\acceptableradii$ has \emph{density-parameter at most} $\density$ if 
\begin{equation}
\sizeof\setof{\radius \in \radii : \radius \leq \bigradius} 
\lesssim_\density \bigradius^{\density} 
\end{equation}
as $\bigradius \to \infty$. 
\end{definition}
For instance a lacunary subsequence has density-parameter at most $\epsilon$ for all $\epsilon>0$ while the full sequence $\acceptableradii$ has density at most $\degree$ (when dimension is sufficiently large with respect to degree). 

\begin{remark}
We mention that our density-parameter a discrete version of the Minkowski dimension of the set of radii considered in \cite{Duoandikoetxea_Vargas} and \cite{Seeger_Wainger_Wright}. 
We note that, with simpler proofs, our discrete density-parameter may be used to recover many of the results in \cite{DRdF_lacunary}, \cite{Duoandikoetxea_Vargas} and \cite{Seeger_Wainger_Wright}. 
We do not explore the relationship between these quantities in this paper. 
\end{remark}

First, we split our operator into \emph{narrow} and \emph{wide averages}: for any $\bigradius>0$, 
\[
\maxop \fxn
\leq \sup_{\radius \leq \bigradius} \absolutevalueof{\avgop \fxn} 
    + \sup_{\radius > \bigradius} \absolutevalueof{\avgop \fxn} 
. 
\]
Since each average decomposes into a main term and an error term: $\avgop = \HLop + \error$, we further decompose the wide averages into:
\[
\sup_{\radius > \bigradius} \absolutevalueof{\avgop \fxn} 
\leq \sup_{\radius > \bigradius} \absolutevalueof{\HLop \fxn} + \sup_{\radius > \bigradius} \absolutevalueof{\error \fxn} 
. 
\]
Altogether, we have the following decomposition lemma. 
\begin{lemma}\label{lemma:decomposition}
For any subsequence $\radii$ of $\acceptableradii$, we can bound the $\degree$-spherical maximal operator over $\radii$ by 
\begin{equation}
\maxop \fxn 
\leq \sup_{\radius \leq \bigradius} \absolutevalueof{\avgop \fxn} + \sup_{\radius > \bigradius} \absolutevalueof{\HLop \fxn} + \sup_{\radius > \bigradius} \absolutevalueof{\error \fxn} 
\end{equation}
where the supremuma are understood to only consider radii $\radius \in \radii$. 
\end{lemma}
The narrow averages are handled by the union bound. 
\begin{proposition}\label{proposition:density_ell1_bound}
If $\radii \subseteq \acceptableradii$ has density-parameter at most $\delta$, then 
\begin{equation}
\lpnorm{1}{\sup_{\radius \leq \bigradius} \absolutevalueof{\avgop \fxn}} 
\lesssim {\bigradius}^{\density} \lpnorm{1}{\fxn}
\end{equation} 
\end{proposition}
\begin{proof}[of Proposition~\ref{proposition:density_ell1_bound}]
Bound the sup by a sum, use Minkowski's theorem to move the $\ell^1(\Z^\dimension)$-norm inside, and then use the definition of density-parameter while noting that $\avgop$ has norm 1 on $\ell^1(\Z^\dimension)$ for every $\radius \in \acceptableradii$. 
%
\end{proof}
%
%

Applying the Approximation Formula and summing \eqref{eq:error_bound} over a geometric series, we obtain the following $\ell^2(\Z^\dimension)$-bound for the error term. 
\begin{proposition}\label{proposition:error_term_bound}
Let $\degree \geq 3$ and \(\radii \subset \acceptableradii\). 
If the dimension $\dimension>\max{\{ \degree(\degree+2), \degree/\powerloss\}}$, then 
\begin{equation}\label{eq:error_term_bound}
\lpnorm{2}{\sup_{\radius > \bigradius} \absolutevalueof{\error \fxn}} 
\lesssim 
{\bigradius}^{-\dyadicsavings} \lpnorm{2}{\fxn} 
\end{equation}
where 
\(
\dyadicsavings 
:= 
\min \{ \dimension \powerloss - \degree, \frac{\dimension}{\degree}-(\degree+2) \} 
> 0 
. 
\)
\end{proposition}

We are now ready for the proof of Theorem~\ref{theorem:main_theorem}. 
\begin{proof}[of Theorem~\ref{theorem:main_theorem}]
Again, take $\indicator{\latticesubset}$ to be the characteristic function of $F$ a subset of $\Z^\dimension$. We will prove the restricted weak-type bound for 
\(
p 
= 
\max \setof{\frac{\dimension}{\dimension-\degree}, \frac{2\dyadicsavings+2\density}{2\dyadicsavings+\density} }
. 
\)
For any altitude $\altitude>0$ and any compactly supported function $\fxn: \Z^\dimension \to \C$, we have 
\begin{align*}
\sizeof{\setof{\absolutevalueof{\maxop \fxn} > \altitude }} 
\leq 
& \sizeof{\setof{ \sup_{\radius \leq \bigradius} \absolutevalueof{\avgop \fxn} > \altitude/3 }} 
  + \sizeof{\setof{ {\sup_{\radius > \bigradius} \absolutevalueof{\error \fxn}} > \altitude/3 }} 
\\
 & \hspace{5mm} + \sizeof{\setof{ {\sup_{\radius > 0} \absolutevalueof{\HLop \fxn}} > \altitude/3 }} 
\end{align*}
Since we are aiming for a restricted weak-type bound, we may assume that $ 0 < \altitude \leq 1$. 
Applying Propostion~\ref{proposition:density_ell1_bound}, we find that 
\begin{equation}\label{eq:sparse_ell1_bound}
\sizeof{\setof{ \sup_{\radius \leq \bigradius} \absolutevalueof{\avgop \fxn} > \altitude/3 }} 
\lesssim \altitude^{-1} {\bigradius}^{\density} \lpnorm{1}{\fxn} 
. 
\end{equation}
On $\ell^2(\Z^\dimension)$, we use Proposition~\ref{proposition:error_term_bound} and the estimate \eqref{eq:error_term_bound}. 
%
%
Balancing \eqref{eq:sparse_ell1_bound} and \eqref{eq:error_term_bound}, we choose 
\(
\bigradius 
= 
\alpha^{-\frac{1}{2\dyadicsavings+\density}} 
\) 
to find that 
\begin{align*}
\sizeof{\setof{\absolutevalueof{\maxop \indicator{\latticesubset}} > \altitude }} 
& \lesssim \altitude^{-(1+\frac{\density}{2\dyadicsavings+\density})} 
    \inparentheses{\lpnorm{1}{\indicator{\latticesubset}} + \lpnorm{2}{\indicator{\latticesubset}}^2}
\\
   & \hspace{10mm}
    + \sizeof{\setof{ {\sup_{\radius > 0} \absolutevalueof{\HLop \indicator{\latticesubset}}} > \altitude/3 }} 
\\ & \lesssim \altitude^{-(1+\frac{\density}{2\dyadicsavings+\density})} 
    \inparentheses{\lpnorm{1}{\indicator{\latticesubset}} + \lpnorm{2}{\indicator{\latticesubset}}^2} 
\\
   & \hspace{10mm}
    + \altitude^{-\frac{\dimension}{\dimension-\degree}} \lpnorm{\frac{\dimension}{\dimension-\degree}}{\indicator{\latticesubset}}^{\frac{\dimension}{\dimension-\degree}}
\\ & \lesssim \altitude^{-(1+\frac{\density}{2\dyadicsavings+\density})} \sizeof{F} + \altitude^{-\frac{\dimension}{\dimension-\degree}} \sizeof{F} 
\end{align*}
where the second inequality follows from applying Theorem~\ref{theorem:main_term}. 
Since $\altitude \in (0,1]$, we see that the summand corresponding to the larger of the two exponents dominates the other summand. 
\end{proof}
%


The proof of Theorem~\ref{theorem:low_density_improvement_by_mean_values} is identical to the proof of Theorem~\ref{theorem:main_theorem} upon replacing Proposition~\ref{proposition:error_term_bound} with the following improvement. 
\begin{proposition}\label{proposition:single_error_term_bound}
Let $\degree \geq 3$ and assume that \(\meanvaluehypothesis{\moment}\) is true for some \(s>\degree\) and \(\hypothesis{\powerloss}\) is true for some \(\powerloss \in (0,1)\) . 
If \(\radii \subset \acceptableradii\) have dimension at most \( \density \in [0,(\dimension-2\moment)\powerloss) \) and the dimension \(\dimension > \max{\{ 2 \moment, \degree(\degree+2) \}}\), then 
\begin{equation}\label{eq:error_term_bound}
\lpnorm{2}{\sup_{\radius > \bigradius} \absolutevalueof{\error \fxn}} 
\lesssim 
{\bigradius}^{-\singlesavings} \lpnorm{2}{\fxn} 
\end{equation}
where 
\(
\singlesavings 
:= 
\min \{ (\dimension-2\moment)\powerloss-\density, \frac{\dimension}{\degree}-(\degree+2)\} 
\) 
is positive. 
\end{proposition}

\begin{proof}
Suppose that \(\radii\) has density at most \(\delta\). 
On each dyadic scale \(\radii \cap [2^j,2^{j+1})\) apply the union bound, \eqref{equation:dyadic_major_arc_approximation} and \eqref{equation:single_minor_arc_estimate} to conclude that 


\begin{equation}\label{equation:mean_value_error_term}
\lpnorm{2}{\sup_{\radius \in \radii \cap [\dyadicradius, 2\dyadicradius)} \absolutevalueof{\error \fxn}} 
\lesssim 
\dyadicradius^{\delta-(\dimension-2\moment)\powerloss} +  \dyadicradius^{\degree+2-\frac{\dimension}{\degree}} 
\eqsim 
\dyadicradius^{-\singlesavings} 
. 
\end{equation}
Sum over dyadic scales to conclude the proof. 
\end{proof}

\section{Explicit ranges in Theorems~\ref{theorem:main_theorem} and \ref{theorem:low_density_improvement_by_mean_values}, and the connection to the Vinogradov mean value conjectures}\label{section:connection_to_Vinogradov}
%
An important problem in number theory is Waring's problem and related to Waring's problem is Vinogradov's mean value conjectures for which there has been exciting recent progress - see \cite{Wooley_efficient_congruencing1}, \cite{Wooley_efficient_congruencing2}, \cite{Wooley_cubic} and \cite{Wooley_best_in_June2014}. 
In this section we discuss the connection between our hypotheses and Waring's problem via Vinogradov's mean value conjectures and theorems. 

The best current bounds that we may use for our Sup Hypothesis \(\hypothesis{\powerloss}\) and Mean Value Hypothesis \(\meanvaluehypothesis{\moment}\) are due to Wooley. 
Plugging in Wooley's bound for exponential sums of degree $\degree \geq 3$, Theorem 1.5 on p. 5 of \cite{Wooley_efficient_congruencing1}, we see that $\maxop$ is restricted weak-type $(p,p)$ for $p>{\frac{\dimension}{\dimension - \degree^2(\degree-1)}}$. By the Marcinkiewicz interpolation theorem, this implies Theorem~\ref{theorem:Hughes_thesis_bound}. 
Wooley has since improved this bound in \cite{Wooley_best_in_June2014} to: 
\begin{Wooley}[Theorem 7.3 on p. 16 of \cite{Wooley_best_in_June2014}]
If $\degree \geq 3$, then $\hypothesis{\powerloss}$ is true for any \(\powerloss < \inbrackets{2 (\degree-1)(\degree-2)}^{-1}\). 
\end{Wooley}
Consequently the best bounds that we may plug into Theorem~\ref{theorem:main_theorem} yield 
\begin{theorem}\label{theorem:restricted_weak_type_bound_by_Wooley} 
If $\degree \geq 3$ and $\dimension>2\degree(\degree-1)(\degree-2)$, then $\fullmaxop$ is of restricted weak-type $(p,p)$ for $p > \frac{\dimension}{\dimension - \degree(\degree-1)(\degree-2)}$.
\end{theorem}
%

We now connect our Mean Value Hypothesis with Vinogradov's mean value conjectures. 
This passage is well known in the circle method arena. 
For \(\moment, \degree, N \in \N\) define the Vinogradov mean value 
\begin{align}
J_{\moment, \degree}(N) 
& := \int_{\T^\dimension} \absolutevalueof{\sum_{n=1}^N \eof{\xi_1 n + \xi_2 n^2 + \dots + \xi_\degree n^\degree} }^{2\moment} \; d\xi 
\\ & = \# \setof{m,n \in [N]^{\moment} : \sum_{i=1}^\moment n_i^\ell = \sum_{i=1}^\moment m_i^\ell \text{ for } \ell = 1, \dots, \degree} 
\end{align}
where equality holds by Plancherel's theorem. 
%
Vinogradov's mean value theorem and conjectures predict that 
\begin{equation}\label{equation:VMC}
J_{\moment, \degree}(N) 
\lesssim_{\moment, \degree} 
N^\moment + N^{2\moment-\frac{\degree(\degree+1)}{2}} 
\end{equation} 
as \(N \to \infty\) for each (fixed) \(\moment,\degree \in \N\) unless \(\degree=2\) and \(\moment=3\). 
If \(\degree=2\) and \(\moment=3\), then \eqref{equation:VMC} is known to hold with a log-loss (a.k.a \(\epsilon\)-loss). 
There is much work on this problem and it is normally stated with the log-loss: 
 \begin{equation}\label{equation:VMC_epsilon_loss}
J_{\moment, \degree}(N) 
\lesssim_\epsilon 
N^{\moment+\epsilon} + N^{2\moment-\frac{\degree(\degree+1)}{2}+\epsilon} 
\end{equation}
for all \(\epsilon>0\) where the implicit constant may depend on \(\epsilon\), but not on \(N\). 
See \cite{Montgomery} for more information. 
The sharpened version we chose in \eqref{equation:VMC} is for aesthetic reasons in the statement of Theorem~\ref{theorem:low_density_improvement_by_mean_values}. 
See \cite{Wooley_diagonal_equations} for information regarding our sharpened assumption \eqref{equation:VMC}. 
Equation (5.37) on p. 69 of \cite{Vaughan} connects our mean value hypothesis with Vinogradov's mean value conjectures: 
\begin{equation}\label{eq:Vinogradov_to_mean_value_hypothesis}
\int_0^1 \absolutevalueof{\alpha_\radius(t,0)}^{2\moment} \; dt 
\lesssim r^{\frac{\degree(\degree-1)}{2}} \cdot J_{\moment, \degree}(r) 
\end{equation}

Throughout the past century Vinogradov's mean value conjectures received intense interest in number theory with many important advances. 
In particular \eqref{equation:VMC_epsilon_loss} is known to be true for relatively small moments \(\moment\) compared with the degree \(\degree\). 
For the purpose of this paper we only use the most recent progress by T. Wooley. 
By Wooley's work on the Vinogradov mean value conjectures, our mean value hypothesis \(\meanvaluehypothesis{\dimension}\) is true (up to a log-loss) for degrees \(\degree \geq 3\) and dimensions \(\dimension > 2\degree(\degree-1)\). 
\begin{Wooley}[Theorem~4.1 of \cite{Wooley_best_in_June2014}] 
\eqref{equation:VMC_epsilon_loss} holds for \(\degree \geq 3\) and \(\moment \geq \degree(\degree-1)\). \end{Wooley}
This is the best one may hope for when \(\degree=3\). 
This implies that 
\begin{equation}\label{eq:Wooley_mean_value}
\int_0^1 \absolutevalueof{\alpha_\radius(t,0)}^{2\moment} \; dt 
\lesssim_\epsilon 
\radius^{\frac{\degree(\degree-1)}{2}} \cdot \radius^{2\moment-\frac{\degree(\degree+1)}{2}} 
= 
\radius^{2\moment-\degree+\epsilon} 
\end{equation}
provided that \(\moment \geq \degree^2-1\). 
Therefore we deduce that 
\begin{equation*}
\lpnorm{2}{A_\radius^{minor}} 
\lesssim_\epsilon 
  \radius^{\degree-\dimension} 
  \cdot \radius^{2\moment-\degree+\epsilon} 
  \cdot \radius^{\moment(1-\powerloss)} 
= 
\radius^{\epsilon-(\dimension-2\moment)\powerloss} 
\end{equation*}
for \(\dimension > 2\moment\) and \(\moment \geq \degree(\degree-1)\). 
We need \((\dimension-2\moment)\powerloss > 0\) so that we may choose \(\epsilon>0\) arbitrarily small.
This bound is maximized by taking \(\moment = \degree(\degree-1)\), in which case we have 
\begin{equation}\label{equation:lacunary_mean_value_minor_arc_bound}
\lpnorm{2}{A_\radius^{minor}} 
\lesssim_\epsilon 
\radius^{\epsilon-(\dimension-2\degree(\degree-1))\powerloss} 
\end{equation}
for \(\dimension > 2\degree(\degree-1)\) and all \(\epsilon>0\). 
Following the proof of Theorem~\ref{theorem:low_density_improvement_by_mean_values} we deduce the following theorem. 
\begin{theorem}\label{theorem:low_density_improvement_by_Vinogradov}
Let \(\dimension > 2\degree(\degree-1)\) and \(\powerloss < \inbrackets{2 (\degree-1)(\degree-2)}^{-1}\). 
If \( \radii \) is a subsequence of \(\acceptableradii\) with density-parameter at most 
\( 
\density \in (0,\powerloss[\dimension-2\degree(\degree-1)]) 
\), 
then when \(\dimension > 2\degree(\degree-1)\), \(\maxop\) is bounded from \(\ell^{p,1}(\Z^\dimension)\) to \(\ell^{p,\infty}(\Z^\dimension)\) for 
\(
p = 
\max 
\{
\frac{\dimension}{\dimension-\degree}, 
1 + \frac{\density\degree}{2(\dimension-\degree[\degree+2])+\density\degree},  
1+\frac{\density}{2\powerloss[\dimension-2\degree(\degree-1)] - \density} 
\}
\) . 
\end{theorem}
Observe that, in contrast to Theorem~\ref{theorem:restricted_weak_type_bound_by_Wooley}, the \(\epsilon\)-loss in \eqref{equation:lacunary_mean_value_minor_arc_bound} does not allow us to capture the endpoint \(\density = \powerloss[\dimension-2\degree(\degree-1)]\) for our density-parameter. 

We have the following corollary. 
\begin{corollary}\label{corollary:mean_value_thin_sequences_range}
Suppose that \(\degree \in \N \) is at least 3. 
If \( \radii \) is a subsequence of \(\acceptableradii\) with density-parameter at most 
\( 
\density 
\leq 
\frac{\dimension-2\degree(\degree-1)}{2(\degree-1)(\degree-2)[\dimension-2\degree]} 
\), then \( \maxop \) is bounded from \(\ell^{p,1}(\Z^\dimension)\) to \(\ell^{p,\infty}(\Z^\dimension)\) for 
\(
p = 
\frac{\dimension}{\dimension-\degree}
\) 
and \(\dimension > 2\degree(\degree-1)\). 
\end{corollary}
\begin{remark}
One should note that for sufficiently thin subsequeunces of \(\acceptableradii\) the dimensional constraint for its maximal function is reduced from a cubic dependence on the degree in Theorem~\ref{theorem:restricted_weak_type_bound_by_Wooley} to a quadratic one in Corollary~\ref{theorem:low_density_improvement_by_Vinogradov}. 
\end{remark}

\begin{remark}
The optimal exponent in \eqref{equation:VMC} is \(\moment = \frac{\degree(\degree+1)}{2}\) when the two summands balance. 
Plugging this into \eqref{eq:Vinogradov_to_mean_value_hypothesis} we predict that 
\begin{equation*}
\lpnorm{2}{A_\radius^{minor}} 
\lesssim 
  \radius^{\degree-\dimension} 
  \cdot \radius^{\degree(\degree+1)-\degree} 
  \cdot \radius^{[\dimension-\degree(\degree+1)](1-\powerloss)} 
= 
\radius^{-(\dimension-\degree[\degree+1])\powerloss} 
\end{equation*}
holds for \(\dimension > \frac{\degree(\degree+1)}{2}\) provided hypothesis \(\hypothesis{\powerloss}\) is true. 
%
Therefore, a resolution of the Vinogradov Mean Value Conjecture would improve the range of \(\ell^p(\Z^\dimension)\)-spaces in Theorem~\ref{theorem:low_density_improvement_by_Vinogradov} to \(\dimension > \degree(\degree+1)\) with the appropriate conditions on \(p\) and \(\powerloss\). 
\end{remark}

%


%
%
\bibliographystyle{amsalpha}
\bibliography{references_weak_vinogradov}
\end{document}